\documentclass[a4paper,12pt]{amsart}
\usepackage{amssymb}
\usepackage{amsmath}
\usepackage{stmaryrd}
\usepackage{amscd,amsthm,amssymb}
\usepackage{enumerate}
\usepackage{color}

\scrollmode
\usepackage{latexsym}

\def\.{\cdot}

\def\vs{\vskip .6cm}

\def\la{\langle}
\def\ra{\rangle}

\def\n{\nabla}

\def\beq{\begin{equation}}
\def\eeq{\end{equation}}
\def\bea{\begin{eqnarray*}}
\def\eea{\end{eqnarray*}}
\def\beaa{\begin{eqnarray}}
\def\eeaa{\end{eqnarray}}
\def\ba{\begin{array}}
\def\ea{\end{array}}

\def \RM{\mathbb{R}}

\def \CM{\mathbb{C}}

\def \SM{\mathbb{S}}

\def\Ric{\mathrm{Ric}}
\def\id{\mathrm{id}}
\def\be{\begin{equation}}
\def\ee{\end{equation}}
\def\tr{\mathrm{tr}}

\def\sp{\mathfrak{sp}}

\def\SU{\mathrm{SU}}

\def\A{\mathrm{A}}

\def\G{\mathrm{G}}

\def\R{\mathrm{R}}

\def\End{\mathrm{End}}

\def\vol{\mathrm{vol}}

\def\Cl{\mathrm{Cl}}
\def\Sp{\mathrm{Sp}}

\def\res{\arrowvert}

\def\scal{\mathrm{scal}}

\def\P{\mathrm{P}}
\def\T{\mathrm{T}}

\newtheorem{epr}{Proposition}[section]
\newtheorem{ath}[epr]{Theorem}
\newtheorem{elem}[epr]{Lemma}
\newtheorem{ecor}[epr]{Corollary}

\theoremstyle{definition}

\newtheorem{ere}[epr]{Remark}

\title[Generalized Killing spinors]{Generalized Killing spinors on spheres}

\author{Andrei Moroianu, Uwe Semmelmann}

\address{Andrei Moroianu \\ Universit\'e de Versailles-St Quentin \\
Laboratoire de Math\'ematiques \\ UMR 8100 du CNRS\\
45 avenue des \'Etats-Unis\\
78035 Versailles, France }
\email{andrei.moroianu@math.cnrs.fr}

\address{Uwe Semmelmann\\
Institut f\"ur Geometrie und Topologie \\
Fachbereich Mathematik\\
Universit{\"a}t Stuttgart\\
Pfaffenwaldring 57 \\
70569 Stuttgart, Germany
}
\email{uwe.semmelmann@mathematik.uni-stuttgart.de}

\date{\today}
\thanks{This work was done during a ``Research in Pairs'' stay at CIRM, Luminy. We warmly thank the CIRM for hospitality. The first author was partially supported by the contract ANR-10-BLAN 0105 ``Aspects Conformes de la G{\'e}om{\'e}trie''. }

\begin{document}

\begin{abstract}
We study generalized Killing spinors on round spheres $\SM^n$. We show that on the standard sphere
$\SM^8$ any generalized Killing spinor has to be an ordinary Killing spinor.
Moreover we classify generalized Killing spinors on $\SM^n$ whose associated symmetric
endomorphism has at most two eigenvalues and recover in particular Agricola--Friedrich's canonical spinor on $3$-Sasakian manifolds of dimension $7$. 
Finally we show that it is not possible to
deform Killing spinors on standard spheres into genuine generalized Killing spinors.
\vs
\noindent 2010 {\it Mathematics Subject Classification}: Primary: {53C25, 53C27, 53C40}
\smallskip

\noindent {\it Keywords}: {generalized Killing
  spinors, parallel spinors.} 
\end{abstract}

\maketitle

\section{Introduction}

A {\it generalized Killing spinor} on a spin manifold $(M,g)$ is a non-zero spinor $ \Psi \in \Gamma(\Sigma M)$
satisfying for all vector fields $X$ the equation $\nabla_X\Psi=A(X)\cdot\Psi$, where $A$ is some symmetric 
endomorphism field. 
If $A$ is a non-zero multiple of the identity, $\Psi$ is called a Killing spinor
\cite{ba,bfgk}. We will call generalized Killing spinors with $A\neq \lambda\id$ {\it genuine} generalized
Killing spinors.

Generalized Killing spinors arise naturally  as the restrictions of parallel spinors on spin manifolds $\hat M$
to hypersurfaces $M \subset \hat M$ (see \cite{bgm,friedrich:98,kf01,morel03,gkse}). In this case the
endomorphism $A$ is half of the second fundamental form of $M$. The converse is true under certain conditions, e.g.
when both the manifold $(M,g)$ and the spinor $\Psi$ are real analytic \cite{amm}.

In low dimensions any generalized Killing spinor $\Psi$ defines a $G$-structure on $M$, where $G$ is the stabilizer of $\Psi$ at some
point. The intrinsic torsion of this $G$-structure is determined by the endomorphism $A$, and since $A$ is assumed to be symmetric, 
some part of the intrinsic torsion has to vanish. This
leads to interesting reformulations of the existence of generalized Killing spinors, e.g. they correspond to half-flat 
$\SU(3)$-structures \cite{chs,hi03} in dimension $6$ and to co-calibrated $\G_2$-structures \cite{cs06,fg} in dimension $7$.

In \cite{gkse} we started an investigation of generalized Killing spinors on Einstein manifolds, motivated by an
analogue of the Goldberg conjecture. We showed that any generalized Killing spinor on the standard spheres 
$\SM^2$ and $\SM^5$, as well as on any $4$-dimensional Einstein manifolds of positive scalar curvature
has to be an ordinary Killing spinor and we have constructed examples of genuine generalized Killing spinors on $\SM^3$.
Moreover, we gave an account of the other examples
of genuine generalized Killing spinors on Einstein manifolds which can be found in the recent literature
 on $\SM^3 \times \SM^3$ and $\CM \P^3$ (cf. \cite{cs06,ms,sh}), and on
$7$-dimensional $3$-Sasakian manifolds (cf.  \cite{af10}). 

In the present article we concentrate on the existence question for generalized Killing spinors on
 standard spheres. It is a classical theorem that any Einstein hypersurface of positive scalar curvature
in the Euclidean space $\RM^{n+1}$ is locally isometric to $\SM^n$. The round spheres are thus the only Einstein hypersurfaces
in $\RM^{n+1}$ admitting generalized Killing spinors. Our problem can be rephrased into the question: {\em Is
it possible to realize  standard spheres as hypersurfaces of  non-flat manifolds with reduced holonomy, e.g.
Calabi-Yau or hyperk\"ahler manifolds?}

Even on such simple manifolds as the standard spheres, the problem of proving 
existence or non existence of genuine generalized Killing spinors turns out to be extremely difficult.
In this article we obtain the following partial results: in Section \ref{sect3} we show that on $\SM^8$ any 
generalized Killing spinor has to be an ordinary Killing spinor. The same statement is true for
any $8k$-dimensional standard sphere if a natural vector field associated to the spinor does not vanish identically.
In Section \ref{sect4} we consider generalized Killing spinors on $\SM^n$ for which the symmetric endomorphism $A$ has
exactly two eigenvalues. We show that this is possible only in dimension $3$ and $7$, where the
generalized Killing spinors
coincide with the examples mentioned above (see also \cite{gh} for similar examples on 3-dimensional Heisenberg manifolds). 

In the last section we investigate deformations
of generalized Killing spinors. Using the Weitzenb\"ock formula for trace-free symmetric tensors we prove a rigidity result for Killing spinors on spheres, similar in some sense with the rigidity of Einstein metrics \cite[Sect. 4.63]{besse}.

\section{Preliminaries}
We refer to \cite{bfgk,lm} for basic definitions in spin geometry and list below some of the most important facts which will be needed in the sequel.
Let $(M^n,g)$ be an $n$-dimensional Riemannian spin manifold with real spinor bundle $\Sigma M$. The
Levi-Civita connection $\nabla$ induces a connection  on $\Sigma M$, also denoted by $\nabla$. In addition the
real spinor bundle $\Sigma M$ is endowed with a $\nabla$-parallel Euclidean scalar product $\la.,.\ra$.

Throughout this article we will identify $1$-forms and bilinear forms with vectors and
endomorphisms respectively, by the help of the Riemannian metric. 

The Clifford multiplication with tangent vectors is parallel with respect to $\nabla$ and skew-symmetric with
respect to $\la.,.\ra$:
\beq\label{skew}
\la X\cdot \Psi,\Phi\ra=-\la\Psi,X\cdot\Phi\ra,\qquad\forall\ X,Y\in\T M,\ \forall\ \Psi,\Phi\in\Sigma M.
\eeq
In particular $\la X\cdot \Psi, \Psi \ra = 0$ for any vector field $X$ and spinor $\Psi$. The Clifford multiplication with $2$-forms is defined via the equation
\beq\label{2f}
(X \wedge Y) \cdot \Psi \;=\; X \cdot Y \cdot \Psi \;+\; g(X, Y ) \, \Psi .
\eeq
Using \eqref{skew} and the basic Clifford formula $X\cdot Y\cdot+Y\cdot X\cdot+2g(X,Y)\id=0$, we easily get 
\beq\label{xy}
\la X\cdot Y\cdot \Psi,\Psi\ra=-g(X,Y)\la\Psi,\Psi\ra,\qquad\forall\ X,Y\in\T M,\ \Psi\in\Sigma M,
\eeq
which together with \eqref{2f} shows that Clifford product with 2-forms is also skew-symmetric.

The  curvature $\R ^{\Sigma M}$ of the spinor bundle and the Riemannian curvature are related by 
\beq\label{curv-0}
\R ^{\Sigma M}_{X, Y}\, \Psi \;=\; \tfrac12 \mathcal R (X \wedge Y) \cdot \Psi \qquad\forall\ X,Y\in
\T M,\ \Psi\in\Sigma M ,
\eeq
where  $\mathcal R : \Lambda^2M \rightarrow \Lambda^2 M$ denotes the curvature operator 
defined by
$$
g( \mathcal R (X \wedge Y), U \wedge V ): = g( \R _{X, Y}U, V),\qquad  \R _{X, Y}:=[\nabla_X,\nabla_Y]-\nabla_{[X,Y]}.
$$ 
Note that with our convention  the curvature operator on the standard sphere acts on 2-forms as minus the identity.

A  {\em generalized Killing spinor}  \cite{amm,bgm,kf01,gkse} on $(M,g)$ is a spinor $\Psi$ satisfying  the
equation 
\beq\label{gks}
\nabla_X\Psi=A(X)\.\Psi, \qquad\forall\  X\in \T M,
\eeq 
where $A\in\Gamma(\End(\T M))$ is  some symmetric endomorphism field, sometimes called the endomorphism {\em associated} to $\Psi$. 
Clearly a generalized Killing spinor $\Psi$ has constant length and by rescaling we may always assume 
that $|\Psi|^2=1$.

After taking a further covariant derivative in Eq. (\ref{gks}) and skew-symmetrizing one obtains
the curvature equation (see \cite[Eq. (9)]{gkse}):
\beq\label{curv}
(d^\nabla A)(X, Y)\; =\; 
 [(\nabla_X A)Y  \,-\, (\nabla_Y A)X  ] \cdot \Psi \;=\; 2\, A(X) \wedge A(Y) \cdot \Psi+\tfrac12 \, \mathcal R (X  \wedge Y) \cdot \Psi.
\eeq
Moreover, one has the following constraint equations (\cite[Eqs. (11) and (12)]{gkse}):
\beq
0   \;=\; \delta^\nabla A \; +  \; d \tr A  ,  \label{three1}
\eeq
\beq
\scal  \;=\; 4 (\tr A)^2 \;-\; 4 \tr A^2 , \label{three2}
\eeq
where $\delta^\nabla A: = - \sum_{i=1}^n (\nabla_{e_i} A)e_i$ denotes the
divergence of $A$.


It is well known that the standard sphere $\SM^n$ admits the maximal possible number of real Killing spinors
trivializing the spinor bundle $\Sigma M$, cf. \cite{ba}. About the existence of generalized Killing spinors much less is
known. We quote the following previous results:
\begin{itemize}
\item There are no genuine generalized Killing spinors on $\SM^2, \SM^4$ and $\SM^5$, cf. \cite{gkse}. 
\item There are examples of genuine generalized Killing spinors on $\SM^3$ of the form
$\Psi = \xi \cdot  \Phi$, where $\xi$ is a unit length left-invariant Killing vector field and $\Phi$ is a Killing spinor with
Killing constant $\frac12$. In this example  the symmetric endomorphism $A$  has eigenvalue $\frac12$
of multiplicity $1$, and eigenvalue $-\frac32$ of multiplicity $2$, cf. \cite{gkse}. 
\item There is a genuine generalized Killing spinor on $\SM^7$, which again is of the form $\Psi = \xi \cdot  \Phi$,
where $\xi$ is a unit length Killing vector field on $\SM^7$ and $\Phi$ is a certain Killing spinor. Like in dimension $3$, the eigenvalues of $A$ are $\frac12$ and $-\frac32$, this time with multiplicities $3$ and $4$, respectively, cf. \cite{af10}.
\end{itemize}

\section{Generalized Killing spinors on $\SM^{8k}$}\label{sect3}

The aim of this section is to show that every generalized Killing spinor on $\SM^{8}$ is a Killing spinor, as well as a partial result in the same direction for all spheres $\SM^{8k}$.

Recall that in dimension $8k$ the real spin representation splits as $\Sigma_{8k} = \Sigma_{8k}^+ \oplus \Sigma_{8k}^-$, where
$\Sigma_{8k}^\pm $ are the $\pm 1$-eigenspaces of 
the multiplication with the volume element and are interchanged by Clifford multiplication with
vectors. Correspondingly,
$\Psi$ splits as $\Psi \;=\; \Psi^+ \;+\; \Psi^-$. Let $\eta$ be the vector field on $\SM^{8k}$ given by 
\beq\label{eta}g(\eta,X) = \la X\cdot\Psi^+,\Psi^-\ra,\qquad\forall \ X\in\T \SM^{8k}.\eeq
If the form $\eta$ does not vanish identically, we have the following:

\begin{ath}\label{s8k}
Let $\Psi$  be a generalized Killing spinor on $\SM^{8k}$. If the one-form defined in \eqref{eta} is non-vanishing on a dense subset, then $\Psi$ is a Killing spinor.
\end{ath}

\begin{proof}
We assume that $\Psi$ is scaled to have unit length.  Denoting $a:=\tr(A)$ and using the fact that the scalar curvature of $\SM^{8k}$ equals $8k(8k-1)$, Eq. \eqref{three2} reads $a^2 - \tr A^2 =2k(8k-1)$. From \eqref{gks} we get:
\beq\label{pm}
\nabla_X \Psi^\pm \;=\; A(X) \cdot \Psi^\mp .
\eeq

Let $S^-$ denote the open set of  points $p \in \SM^{8k}$ with $\Psi^-_p \neq 0$. 
It is easy to see that $S^-$ is dense. Indeed, if $U$ were a non-empty
open subset of $\SM^{8k}\setminus S^-$, then \eqref{pm} yields $A(X)\cdot \Psi^+=0$ for all $X\in\T U$, so
$A|_U=0$. By \eqref{pm} again, $\Psi^+$ is parallel (and non-zero) on $U$, so the Ricci tensor of $\SM^{8k}$ vanishes on $U$, which is absurd.
A similar argument shows that the set $S^+$ where $\Psi^+$ is non-vanishing is also dense, so the set $S:=S^-\cap S^+$ is dense in $\SM^{8k}$.

We denote by $h : =|\Psi^-|^2$ the length function of $\Psi^-$. Since $\Psi$ has unit length, $|\Psi^+|^2=1-h$.
From \eqref{pm}, the derivative of $h$ in the direction of any tangent vector $X$ reads
$$
dh(X) =
2\la \nabla_X \Psi^-, \Psi^-\ra \;=\; 2\la  A(X) \cdot \Psi^+, \Psi^- \ra \;=\; 2\eta(A(X))=2g(A(\eta),X),
$$
whence 
\beq\label{dh} dh=2A(\eta).
\eeq
Taking the covariant derivative in the direction of $Y$ in \eqref{eta}, assuming that $X$ is
parallel at some point and using \eqref{pm} yields
\bea
g( \nabla_Y \eta, X) & =& \la X \cdot A(Y) \cdot \Psi^-, \Psi^- \ra   \;+\; \la X \cdot \Psi^+, A(Y) \cdot \Psi^+ \ra  \\
& = &- g( X, A(Y)) \, |\Psi^-|^2 \;+\; g (X, A(Y)) \,  |\Psi^+|^2   \\
& = & (1-2h) g( A(Y), X) ,
\eea
so 
\beq\label{deta} \nabla_Y\eta=(1-2h)A(Y),\qquad\forall\  Y\in \T\SM^{8k}.
\eeq

Taking the covariant derivative with respect to some vector field $X$ in this equation, using \eqref{dh} and skew-symmetrizing, yields:
$$R_{Y,X}\eta=(1-2h)((\nabla_YA)X-(\nabla_XA)Y)-4g(A(\eta),Y)A(X)+4g(A(\eta),X)A(Y), $$
and since the curvature of the round sphere satisfies $R_{Y,X}Z=g(X,Z)Y-g(Y,Z)X$ for all vectors $X,Y,Z$, we get
$$(1-2h)((\nabla_YA)X-(\nabla_XA)Y)=4g(A(\eta),Y)A(X)-4g(A(\eta),X)A(Y)+g(X,\eta)Y-g(Y,\eta)X.$$
Using this last equation in the curvature equation \eqref{curv} we obtain that for every vectors $X,Y$ the following relation holds:
\beq\label{ceta}\begin{split}
&(2h-1)\left(2A(X)\cdot A(Y) +2g(A(X),A(Y))-\tfrac12 X  \cdot Y-\tfrac12g(X,Y)\right) \cdot \Psi\\=&\left(4g(A(\eta),Y)A(X)-4g(A(\eta),X)A(Y)+g(X,\eta)Y-g(Y,\eta)X\right)\cdot \Psi
\end{split}\eeq
(we have used the well known formula $X\wedge Y=X\cdot Y+g(X,Y)$ and the fact that the curvature endomorphism of the round sphere is minus the identity).

In \eqref{ceta} we take the Clifford product with X and sum over an orthonormal basis $X=e_i$. Using the standard formulas in Clifford calculus this yields
\[\begin{split}&(2h-1)\left(-2aA(Y) +2A^2(Y)+\tfrac{8k-1}2 Y\right) \cdot \Psi\\=&\left(-4ag(A(\eta),Y)-4A(\eta)\cdot A(Y)+\eta\cdot Y+8kg(\eta,Y)\right)\cdot\Psi.
\end{split}\]
Taking the scalar product with $\Psi$ in this formula gives 
$$0=-4ag(A(\eta),Y)+4g(A(\eta), A(Y))+(8k-1)g(\eta,Y),\qquad\forall\  Y\in\T\SM^{8k},$$
whence 
\beq\label{a1} A^2(\eta)=aA(\eta)-\tfrac{8k-1}4\eta.
\eeq
We now take the Clifford product with A(X) in \eqref{ceta} and sum over an orthonormal basis $X=e_i$ to obtain
\[\begin{split}&(2h-1)\left(-2\tr A^2A(Y) +2A^3(Y)+\tfrac{1}2 aY-\tfrac12A(Y)\right) \cdot \Psi\\=&\left(-4\tr A^2g(A(\eta),Y)-4A^2(\eta)\cdot A(Y)+A(\eta)\cdot Y+ag(\eta,Y)\right)\cdot\Psi.
\end{split}\]
Taking again the scalar product with $\Psi$ and using \eqref{three2} yields
$$0=(8k(8k-1)-4a^2)g(A(\eta),Y)+4g(A^2(\eta), A(Y))-g(A(\eta),Y)+ag(\eta,Y),\qquad\forall\  Y\in\T\SM^{8k},$$
whence 
\beq\label{a2} A^3(\eta)=(a^2-2k(8k-1)+\tfrac{1}4)A(\eta)-\tfrac a4\eta.
\eeq
Plugging \eqref{a1} into this equation shows that $A(\eta)=\tfrac1{8k}a\eta$, so from \eqref{a1} again we get 
$$\tfrac{a^2}{64k^2}\eta=\tfrac{a^2}{8k}\eta-\tfrac{8k-1}4\eta.$$
As $\eta$ is non-vanishing on a dense subset, we obtain $a^2=16k^2$ on $\SM^{8k}$. This, together with \eqref{three2}, shows that the square norm of the trace-free symmetric tensor $A-\tfrac{a}{8k}\id$ vanishes:
$$|A-\tfrac{a}{8k}\id|^2=\tr(A-\tfrac{a}{8k}\id)^2=\tr A^2-\tfrac{a}{4k}\tr A+\tfrac{a^2}{8k}=\tr A^2-\tfrac{a^2}{8k}=16k^2-2k(8k-1)-2k=0.$$
This implies that $A=\tfrac a{8k}\id=\pm\tfrac12\id$ and thus finishes the proof.\end{proof}

\begin{ecor}\label{s8}
Every generalized Killing spinor $\Psi$ on $\SM^8$ is a Killing spinor.
\end{ecor}
\begin{proof}
For every $p\in S^+$ the injective map $X \in \T_p \SM^8 \mapsto X \cdot \Psi^+ \in (\Sigma_8^-)_p$ is
bijective since $\dim \T _p \SM^8 = \dim (\Sigma_8^-)_p=8$. Consequently, the vector field $\eta$ is non-vanishing on $S$.
\end{proof}

\section{Generalized Killing spinors with two eigenvalues}\label{sect4}

In this section we consider generalized Killing spinors $\Psi$ on the sphere $(M,g):=\SM^n$ ($n\ge 3$) and assume that the associated symmetric endomorphism
$A$ has at each point at most two eigenvalues $\lambda$ and $\mu$. If these eigenvalues coincide at each point, then it is well known that their common value is constant on $M$, so $\Psi$ is a Killing spinor. We assume from now on that $\lambda \neq \mu$ at least at some point of $M$, and thus on some non-empty contractible open set $S$ (it turns out that they are actually constant on $M$, cf. Lemma \ref{41}). We will denote by  $\T^\lambda  \subset \T M $ and  $\T^\mu  \subset \T M$ the eigenspaces corresponding to $\lambda$ and $\mu$ respectively. These two subspaces are mutually orthogonal at each point and are well-defined distributions on $S$.

We start with calculating the derivative $d^\nabla A$ at points of $S$ in three different cases. First, let $X,Y\in\T^\mu$:
\bea
(\nabla_XA)Y - (\nabla_YA)X & = &  X(\mu)Y+ \mu \nabla_XY  -  A(\nabla_XY) - Y(\mu)X  - \mu \nabla _YX+ A(\nabla_YX)  \\
& =  &  (\mu - \lambda) (\nabla_XY)^\lambda  - (\mu - \lambda) (\nabla_YX)^\lambda  +  X(\mu)Y - Y(\mu)X \\
& = &  (\mu - \lambda)[X,Y]^\lambda   +  X(\mu)Y - Y(\mu)X,
\eea
Where the superscript $\lambda$ denotes the projection of the corresponding vector on $\T^\lambda$. 
A similar calculation for a pair of vectors $U, V\in \T^\lambda$ leads to
$$
(\nabla_VA)U - (\nabla_UA)V = (\lambda - \mu ) [U,V]^\mu  + V(\lambda)U - U(\lambda)V .  
$$
Finally, on a mixed pair of vectors $X\in\T^\mu,\ V\in\T^\lambda$, we find
$$
(\nabla_XA)V - (\nabla_VA)X =  (\lambda - \mu) (\nabla_XV)^\mu  - (\mu - \lambda)(\nabla_VX)^\lambda  - V(\mu)X + X(\lambda)V  .
$$

Substituting the equations above  into the curvature equation (\ref{curv}), with $\mathcal R = -\id$ for the sphere, we obtain for every $X,Y\in\T^\mu$ and $U, V\in \T^\lambda$:
\begin{align}
(2\mu^2 - \tfrac12) \, X \wedge Y \cdot \Psi   & \; =\;  (\mu - \lambda) [X,Y]^\lambda  \cdot \Psi + (X(\mu)Y - Y(\mu)X) \cdot \Psi  \label{curv1},\\[1.5ex]
(2\lambda^2 - \tfrac 12) \, V \wedge U \cdot \Psi &  \; =  \; (\lambda -\mu) [V,U]^\mu  \cdot \Psi + ( V(\lambda)U - U(\lambda)V)\cdot \Psi \label{curv2},\\[1.5ex]
(2\lambda \mu - \tfrac12) \, X \wedge V \cdot \Psi & \;  = \;  (\lambda - \mu) ((\nabla_XV)^\mu   + (\nabla_VX)^\lambda  ) +(X(\lambda)V -V(\mu)X ) \cdot \Psi \label{curv3}.
\end{align}

\bigskip

\begin{elem}\label{41}
If $n\ge 3$, the eigenvalues $\lambda$ and $\mu$ are constant on $\SM^n$.
\end{elem}
\proof
Since the sphere is connected, it is enough to show that $\lambda$ and $\mu$ are constant on the open set $S$. Let $p$ and $q$ denote the dimensions of $\T^\lambda $ and  $\T^\mu $ respectively (which are constant on $S$). The assumption $n\ge 3$ shows that at least one of $p$ and $q$ is larger than 1. Assume for the rest of the proof that $q\ge 2$. 

Taking the scalar product with $U \cdot \Psi$ in equation (\ref{curv2}) for $U\in\T^\lambda$ orthogonal to $V\in\T^\lambda$ implies that 
\beq\label{en} V(\lambda)=0,\qquad\forall V\in \T^\lambda.\eeq
On the other hand, (\ref{three2}) reads
\beq\label{elm}
(p\lambda + q \mu)^2 - (p\lambda^2 + q \mu^2) = \tfrac14 n(n-1).
\eeq
Differentiating this relation with respect to some vector $V\in\T^\lambda$ and using \eqref{en} gives 
$$
V(\mu) (\mu (q-1) + p \lambda) = 0.$$
Assuming that $V(\mu)$ is different from zero on some open set $S'\subset S$, then  
\beq\label{elm1}\mu (q-1) + p \lambda = 0\eeq 
on $S'$. 
Differentiating again with respect to $V\in \T^\lambda$ and using \eqref{en}, we get 
$(q-1)V(\mu)=0$.
The assumption 
that $V(\mu)$ is different from zero on $S'$ implies that $q=1$, which contradicts our assumption $q\ge 2$. Thus $V(\mu)=0$ for all $V\in\T^\lambda$ at each point of $S$.

If $p\ge 2$, a similar argument shows that $X(\lambda)=0$ and $X(\mu)=0$ for every $X\in\T^\lambda$, so $\lambda$ and $\mu$ are constant.

It remains to study the case $p=1$. Then $\T^\mu$ is spanned by a unit vector field $X$ on $S$. If $\lambda$ is not constant on $S$, there exists a non-empty open subset $S'\subset S$ such that $d\lambda\ne 0$ on $S'$. Using \eqref{en} we obtain that $X(\lambda)\ne 0$ on $S'$. On the other hand, taking a further derivative in \eqref{en} and skew-symmetrizing yields $0=[U,V](\lambda)$ for every vector fields $U,V$ tangent to $\T^\lambda$. From \eqref{en} again, this implies $[U,V]^\mu(\lambda)=0$, whence $[U,V]^\mu=0$ at each point of $S'$. Using \eqref{curv2} for $U$ orthogonal to $V$ and both non-zero, yields $4\lambda^2=1$ on $S'$, so $d\lambda=0$ on $S'$, contradicting the definition of $S'$. Thus $\lambda$ is constant on $S$, and by \eqref{elm} $\mu$ is also constant on $S$ since it satisfies a second order polynomial equation with constant coefficients and non-vanishing leading coefficient $q(q-1)$.
\qed

\begin{elem}\label{12}
One of the eigenvalues $\lambda$ and $\mu$ has to be equal to $\pm\frac12$.
\end{elem}
\proof 
Assume first that $\lambda\mu = \frac14$. Then the right hand side of (\ref{curv3}) vanishes, so by Lemma \ref{41} we get $(\nabla_XV)^\mu + (\nabla_VX)^\lambda =0$ for every vector fields $X$ and $U$ tangent to $\T^\mu$ and $\T^\lambda$ respectively. Since $\T^\lambda$ and $\T^\mu$ are orthogonal, this shows that $(\nabla_XV)^\mu = 0 = (\nabla_VX)^\lambda $. Thus $\T^\lambda $ and $\T^\mu $
are two non-trivial parallel distributions on $\SM^n$, which is clearly a contradiction. Consequently, $\lambda\mu \ne \frac14$.

Since even-dimensional spheres do not have any non-trivial distributions, it follows that $n=2k+1$ is odd.
By changing the notations if necessary, we can assume that $\dim(\T^\mu )>\dim(\T^\lambda )$. If $\mu^2=\frac14$ we are done, so for the remaining part of the proof we assume that $\mu^2\ne \frac14$. From \eqref{curv1} it follows that for every $x\in\SM^n$ and $X,Y\in\T_x^\mu $ with $X\perp Y$, the vector $[X,Y]^\lambda $ is non-zero (note that this expression is tensorial in $X$ and $Y$, so it only depends on their values at $x$). Consequently, the map $Y\mapsto [X,Y]^\lambda $ from the orthogonal complement of $X$ in $\T_x^\mu $ to $\T_x^\lambda $ is injective. From the dimensional assumption it follows that $\dim(\T_x^\mu )=k+1$ and $\dim(\T_x^\lambda )=k$, so in particular the above map is bijective. It follows that for every $X\in \T_x^\mu $ and $V\in\T_x^\lambda $ there exists a unique $Y\in\T_x^\mu $, $Y\perp X$, such that $[X,Y]^\lambda =V$. Applying \eqref{curv1} and \eqref{curv3} to these vectors yields
\bea(\lambda - \mu) ((\nabla_XV)^\mu   + (\nabla_VX)^\lambda  )\cdot\Psi&=&(2\lambda \mu - \tfrac12) \, X \cdot V \cdot \Psi =
(2\lambda \mu - \tfrac12) \, X \cdot [X,Y]^\lambda  \cdot \Psi\\&=&\tfrac1{\mu-\lambda}(2\lambda \mu - \tfrac12) (2\mu^2 - \tfrac12)\,X\cdot X\cdot Y \cdot \Psi
\\&=&-\tfrac{|X|^2}{\mu-\lambda}(2\lambda \mu - \tfrac12) (2\mu^2 - \tfrac12) \, Y \cdot \Psi.
\eea
This shows that for every $X\in \T_x^\mu $ and $V\in\T_x^\lambda $, the vector $(\nabla_VX)^\lambda $ vanishes, thus $\T^\lambda $ is a totally geodesic distribution. From \eqref{curv2} we deduce that $\lambda^2=\frac14$ unless $k=1$. It remains to rule out the case where $n=3$. 

In this case $\T^\lambda $ is one-dimensional, so we can consider a unit vector $V$ which spans it at each point. Then $V$ is geodesic and taking the scalar product with $X\cdot\Psi$ in \eqref{curv3} shows that $g(\nabla_XV,X)=0$ for every $X\in\T^\mu $. Thus $V$ is a unit Killing vector field on $\SM^3$. It is well known that every such vector satisfies $|\nabla_XV|=|X|$ for every $X$ orthogonal to $V$. Comparing the norms of the two spinors in \eqref{curv3} yields
$2\lambda\mu-\frac12=\pm(\lambda-\mu)$, which can be rewritten as $(2\lambda\pm1)(2\mu\mp1)=0$. This proves the lemma.
\qed

\noindent

Up to a change of orientation we thus may from now on assume that $\lambda = \frac12$.

\begin{elem}
The distribution $\T^\lambda $ is totally geodesic. Moreover, the following equations hold
for any vectors $X,Y\in\T^\mu$ and $V\in\T^\lambda$:
\begin{align}
(2\mu+1) X \wedge Y \cdot \Psi & \;= \;  [X,Y]^\lambda  \cdot \Psi \label{curv1a},\\[1.5ex]
X \cdot V \cdot \Psi  & \; = - (\nabla_XV)^\mu    \cdot \Psi \label{curv3a}.
\end{align}
\end{elem}
\proof
We have $\lambda = \frac12$ and $\mu \neq \lambda$ constant. Equation \eqref{curv1a} thus follows directly from \eqref{curv1}.

Next, taking in (\ref{curv3})  the scalar product
with $V \cdot \Psi$, gives $0 = g((\nabla_VX)^\lambda , V) = - g(X, \nabla_VV) ,$
and by polarization $(\nabla_VU + \nabla_UV)^\mu $ vanishes for every vector fields $U,V$ in $\T^\lambda $. On the other hand, \eqref{curv2} implies $[V,U]^\mu =0$, so adding these two relations we obtain that $(\nabla_UV)^\mu =0$, i.e. $\T^\lambda $ is totally geodesic. 

In particular this can also be expressed by the fact that $(\nabla_VX)^\lambda $ vanishes for every $X\in\T^\mu $ and $V\in\T^\lambda $, so 
\eqref{curv3a} follows directly from \eqref{curv3}.
\qed

\begin{ere}\label{rem}
With  a  similar argument we get
$(\nabla_XY + \nabla_YX)^\lambda =0$ for all vectors $X,Y\in\T^\mu$. Thus the distribution  $\T^\mu $ would also be totally geodesic if integrable.
\end{ere}

\begin{ecor}\label{cliff}
For every $x\in\SM^n$ there is a  representation of the real Clifford algebra $\Cl(\T_x^\lambda )$ on   $\T_x^\mu $.
\end{ecor}
\proof
For $V\in\T_x^\lambda $ and $X\in\T _x^\mu $ we define
$$
\rho_V ( X)  := (\nabla_XV)^\mu .
$$
Then (\ref{curv3a}) can be re-written as $\rho_V(X)\cdot\Psi=V\cdot X\cdot \Psi$, whence
$$
(\rho_V \circ\rho_V( X)) \cdot \Psi  =   V\cdot \rho_V(X) \cdot \Psi=  V\cdot V\cdot X \cdot \Psi= -|V|^2X \cdot \Psi ,
$$
showing that $\rho_V \circ\rho_V=-|V|^2\id$. This proves the lemma.
\qed

\begin{elem}
The second eigenvalue of $A$ is $\mu = - \frac32$.
\end{elem}
\proof
Taking  in (\ref{curv1a}) the scalar product with $V\cdot \Psi$ and applying (\ref{curv3a}), gives
\bea
g([X,Y], V) & = & - (2\mu + 1) \la V\cdot X \cdot Y \cdot \Psi, \Psi \ra = - (2\mu +1) \la X\cdot V \cdot \Psi , Y\cdot \Psi\ra  \\
&= & - (2\mu+1) g(\nabla_XY, V)
\eea
This equation can be rewritten as $g((2\mu + 2) \nabla_XY - \nabla_YX, V) = 0$. Interchanging $X$ and $Y$ 
and subtracting the resulting equations we obtain
$(2\mu + 3) [X,Y]^\lambda =0$. 

If $\mu \neq -\frac32$, 
the distribution $\T^\mu $ is totally geodesic (see Remark \ref{rem}),
and since $\T^\mu $ is also totally geodesic, both distributions would be parallel, which is of course impossible on $\SM^n$.
\qed

\begin{elem}\label{mult}
The multiplicities $p$ and $q$ of $\lambda$ and $\mu$ are related by $q= p+1$.
\end{elem}
\proof 
Introducing the values $\lambda = \frac12$ and $\mu = -\frac32$ in  (\ref{three2}) we obtain the equation
$$
\tfrac14 n(n-1) = a^2 - \tr A^2 =  (\tfrac{p}{2} - \tfrac{3q}{2})^2 -\tfrac{p}{4} - \tfrac{9q}{4} .
$$
Substituting $n = p+q$ immediately leads to  $p=q-1$.
\qed

\begin{ecor}
The pair $(p,q)$ of multiplicities of $\lambda$ and $\mu$ is one of $(1,2), (3,4)$ or $(7,8)$.
\end{ecor}
\proof By Corollary \ref{cliff} and Lemma \ref{mult}, there exists a $\Cl_p$ representation
on $\RM^{p+1}$. From the classification of real Clifford algebras
(cf. \cite{lm}), this can only happen when $p$ is $1$, $3$ or $7$.
\qed

We thus see that a generalized Killing spinor whose associated endomorphism has two eigenvalues can only exist on $\SM^n$ for $n=3$, $n=7$ or $n=15$.  We will now further investigate the geometry determined by $\Psi$ and at the end we will consider these three cases separately.

For every $V\in\T^\lambda $ consider the skew-symmetric endomorphism  $\rho_V$ of $\T^\mu $ defined above by $\rho_V(X):=-(\nabla_XV)^\mu $. Equation \eqref{curv3a} then reads 
\beq X \cdot V \cdot \Psi   = \rho_V(X)   \cdot \Psi \label{curv3b},\qquad\forall\  X\in\T^\mu ,\ \forall\  V\in\T^\lambda  .\eeq
For every $U,V\in\T^\lambda $ with $g(U,V)=0$ we pick some arbitrary vector $X\in\T^\mu $ with $|X|=1$ and write using \eqref{curv1a} and \eqref{curv3b}:
\bea U\cdot V\cdot\Psi&=&(X\cdot U)\cdot(X\cdot V)\cdot\Psi=(X\cdot U)\cdot\rho_{V}(X)\cdot\Psi=\rho_{V}(X)\cdot(X\cdot U)\cdot\Psi\\
&=&\rho_{V}(X)\cdot\rho_{U}(X)\cdot\Psi\in\T^\lambda \cdot\Psi.\eea
This shows that $\Lambda^2\T^\lambda \cdot\Psi\subset\T^\lambda \cdot\Psi$. Moreover, this also shows that for every $X\in\T^\mu $ and $U,V\in\T^\lambda $
\beq\label{sc} \la U\cdot V\cdot\Psi,X\cdot\Psi\ra=0.
\eeq

\begin{elem}
The sub-bundle $\T^\lambda \cdot\Psi$ of $\Sigma\SM^n$ is parallel with respect to the modified connection $\tilde\nabla_X:=\nabla_X-\frac12 X\cdot$. 
\end{elem}
\begin{proof}
For $X\in\T^\mu $  and $V\in\T^\lambda $ we have
\bea (\n_X-\tfrac12X\cdot )(V\cdot\Psi)&=& (\n_XV)\cdot\Psi+V\cdot\A(X)\cdot\Psi-\tfrac12X\cdot V\cdot\Psi\\&=& (\n_XV)\cdot\Psi-\tfrac32V\cdot X\cdot\Psi-\tfrac12X\cdot V\cdot\Psi\\
&=&(\n_XV)\cdot\Psi-V\cdot X\cdot\Psi=(\n_XV)\cdot\Psi+\rho_V(X)\cdot\Psi\\
&=&(\n_XV)^\lambda \cdot\Psi\in \T^\lambda \cdot\Psi,
\eea
and for $U,V\in\T^\lambda $, keeping in mind that $\T^\lambda $ is totally geodesic and that $\Lambda^2\T^\lambda \cdot\Psi\subset\T^\lambda \cdot\Psi$:
\bea (\n_U-\tfrac12U\cdot )(V\cdot\Psi)&=& (\n_UV)\cdot\Psi+V\cdot\A(U)\cdot\Psi-\tfrac12U\cdot V\cdot\Psi\\&=& 
 (\n_UV)\cdot\Psi+\tfrac 12V\cdot U\cdot\Psi-\tfrac12U\cdot V\cdot\Psi\\&=&  (\n_UV)\cdot\Psi+V\wedge U\cdot\Psi
\in \T^\lambda \cdot\Psi.
\eea
\end{proof}

Since $\tilde\nabla$ is flat on $\Sigma\SM^n$, it follows that $\T^\lambda \cdot\Psi$ can be trivialized with $\tilde\nabla$-parallel (i.e. Killing) spinors.
We denote by $\mathcal{K}$ the $p$-dimensional vector space of Killing spinors on $\SM^n$ obtained in this way.
By definition, for every $\Phi\in \mathcal{K}$, there exists a vector field $\xi_\Phi\in\T^\lambda $ satisfying $\xi_\Phi\cdot\Psi=\Phi$. Clearly $\la\Psi,\Phi \ra=0$, and as $\Psi$ has unit norm, $|\xi_\Phi|^2=|\Phi|^2$.
For every tangent vector $X$ we have $g(\xi_\Phi,X)=\la X\cdot\Psi,\Phi\ra$. Using the obvious fact that $A(X)^\lambda =\frac12 X^\lambda $ and $A(X)^\mu =-\frac32 X^\mu $, we compute using \eqref{sc}:
\bea g(\nabla_X\xi_\Phi,X)&=&\la X\cdot\nabla_X\Psi,\Phi\ra+\la X\cdot\Psi,\nabla_X\Phi\ra=\la X\cdot A(X)\cdot\Psi,\Phi\ra+\tfrac12\la X\cdot\Psi,X\cdot\Phi\ra\\
&=&\la X\cdot A(X)\cdot\Psi,\Phi\ra=\la (X^\mu +X^\lambda )\cdot(\tfrac12X^\mu -\tfrac32X^\lambda )\cdot\Psi,\xi_\Phi\cdot\Psi\ra\\
&=&-\tfrac32\la X^\mu \cdot X^\lambda \cdot\Psi,\xi_\Phi\cdot\Psi\ra+\tfrac12\la X^\lambda \cdot X^\mu \cdot\Psi,\xi_\Phi\cdot\Psi\ra=0.
\eea
This shows that $\xi_\Phi$ is a Killing vector field on $\SM^n$ for every Killing spinor $\Phi\in \mathcal{K}$. There exists thus a linear map $F$ from $\mathcal{K}$ to $\Lambda^2\RM^{n+1}$ which associates to each $\Phi\in\mathcal{K}$ a skew-symmetric matrix $F_\Phi\in\Lambda^2\RM^{n+1}$ such that $(\xi_\Phi)_x=F_\Phi(x)$ for every $x\in\SM^n\subset\RM^{n+1}$. In fact $F_\Phi$ is related to the covariant derivative of $\xi_\Phi$ by
\beq\label{fxi}\nabla_X\xi_\Phi=F_\Phi(X),\qquad\forall\  X\in \T\SM^n.
\eeq
As $|\xi_\Phi|^2=|\Phi|^2$, we obtain $(F_\Phi)^2=-|\Phi|^2\id_{\RM^{n+1}}$. 
If we choose now an orthonormal basis $\Phi_1,\ldots,\Phi_p$ of $\mathcal{K}$, and denote by $F_i:=F_{\Phi_i}$ for simplicity, the previous relation becomes
\beq\label{sas}(F_i)^2=-\id,\qquad F_i\circ F_j+F_j\circ F_i=0\ \hbox{for }i\ne j.\eeq
We now consider the three cases above separately.

{\bf The case $n=3$.} In this case the distribution $\T^\lambda $ is $1$-dimensional, and the unit vector field generating it (unique up to a sign) is Killing. The symmetric tensor $A$ thus coincides with the one defined in \cite[Sect. 4.2]{gkse}. Of course, the space of generalized Killing spinors with respect to this tensor $A$ is 4-dimensional, since the spin representation in dimension 3 has a quaternionic structure.

{\bf The case $n=7$.} We have seen that $\{\xi_1,\xi_2,\xi_3\}$ is an orthonormal basis of $\T^\lambda $ at each point consisting of unit Killing vector fields. It is well known that every unit Killing vector field on the round sphere is Sasakian. The relation \eqref{sas} just tells that the triple $\{\xi_1,\xi_2,\xi_3\}$ defines a 3-Sasakian structure.

We remark that the spinor $\Psi$ is exactly the {\em canonical spinor} constructed by Agricola and Friedrich \cite{af10} on any 3-Sasakian manifold of dimension 7.

{\bf The case $n=15$.} It would have been interesting to obtain examples of generalized Killing spinors with two eigenvalues on $\SM^{15}$ similar to those constructed above in dimension 3 and 7. Unfortunately this turns out to be impossible. 

Assuming the existence of such a spinor $\Psi$, we would obtain from the construction above an orthonormal set of Killing vector fields $\xi_1,\ldots,\xi_7$ on $\SM^{15}$ whose defining endomorphisms $F_i\in\Lambda^2\RM^{16}$ satisfy \eqref{sas}. This shows that there exists a representation of the real Clifford algebra $\Cl_7$ on $\RM^{16}$ such that $F_i(x)=e_i\cdot x$ for every $x\in\RM^{16}$ and $1\le i\le 7$. By definition of $F_i$ we thus have $(\xi_i)_x=e_i\cdot x$ for every $x\in\SM^{15}$ and $1\le i\le 7$. As $\Cl_7=\RM(8)\oplus\RM(8)$, this representation decomposes in a direct sum $\RM^{16}=\Sigma_1\oplus\Sigma_2$ of two 8-dimensional representations of $\Cl_7$. Each $x_i\in\Sigma_i$ ($i\in\{1,2\}$) defines a vector cross product $P_{x_i}$ on $\RM^7$ by the formula $(u\wedge v)\cdot x_i=P_{x_i}(u,v)\cdot x_i$. 

Using \eqref{fxi} we can write for every $x=(x_1,x_2)\in \SM^{15}$ and 
$i\ne j\in\{1,\ldots,7\}$:
\bea(\nabla_{\xi_i}\xi_j)_x&=&F_j(\xi_i)_x=F_j(F_i(x))=e_j\cdot e_i\cdot x=(e_j\wedge e_i\cdot x_1,e_j\wedge e_i\cdot x_2)\\
&=&(P_{x_1}(e_j,e_i)\cdot x_1,P_{x_2}(e_j,e_i)\cdot x_2).\eea
Recall now that $\xi_1,\ldots,\xi_7$ span a totally geodesic distribution on $\SM^{15}$. This implies that there exist functions $f_1,\ldots,f_7$ on $\SM^{15}$ such that 
$$(\nabla_{\xi_i}\xi_j)_x=\sum_{k=1}^7f_k(x)(\xi_k)_x=\sum_{k=1}^7f_k(x)F_k(x)=\sum_{k=1}^7f_k(x)e_k\cdot x=\sum_{k=1}^7f_k(x)(e_k \cdot x_1, e_k \cdot x_2).$$
Comparing these last two equations yields $P_{x_1}(e_j,e_i)=P_{x_2}(e_j,e_i)$ for every $(x_1,x_2)\in\SM^{15}\subset\RM^{16}$ and for every 
$i\ne j\in\{1,\ldots,7\}$ . This implies that the vector cross product $P_x$ is independent of $x$, which is of course a contradiction. There are thus no solutions on the sphere $\SM^{15}$.

We have proved the following 

\begin{ath}\label{2eig} Let $\Psi$ be a generalized Killing spinor on the sphere $\SM^n$ whose associated symmetric endomorphism
$A$ has at most two eigenvalues $\lambda$ and $\mu$ at each point. Then $\lambda$ and $\mu$ are both constant. If $\lambda=\mu$, then $A=\pm\frac12\id$ and $\Psi$ is a Killing spinor. If $\lambda\ne\mu$, then up to a permutation of $\lambda$ and $\mu$ and a change of orientation one has $\lambda=\frac12$, $\mu=-\frac32$ and $n=3$ or $n=7$. 
\begin{itemize}
\item If $n=3$, the $\frac12$-eigenspace of $A$ is spanned by a unit left-invariant Killing vector field $\xi$ on $\SM^3$ and $\Psi=\xi\cdot\Phi$ for some Killing spinor $\Phi$ with constant $\frac12$.  
\item If $n=7$, the $\frac12$-eigenspace of $A$ is spanned by three Killing vector fields $\xi_1,\xi_2,\xi_3$ defining a $3$-Sasakian structure on $\SM^7$ and $\Psi$ is the canonical spinor of the $3$-Sasakian structure introduced in \cite{af10}.
\end{itemize}
\end{ath}

\section{Deformations of generalized Killing spinors}

In this section we study the deformation problem for generalized Killing spinors on spheres, and show in particular that Killing spinors are rigid, in the sense that they cannot be deformed into generalized Killing spinors.

For every spin manifold $(M,g)$, the set $\mathcal{GK}(M,g)$ of generalized Killing spinors is a Fr\'echet manifold. On the round sphere $\SM^n$, the (finite dimensional) vector spaces $\mathcal{K}_{\frac12}(\SM^n)$ and $\mathcal{K}_{-\frac12}(\SM^n)$ consisting of Killing spinors with Killing constants $\pm\tfrac12$ respectively, are Fr\'echet submanifolds of $\mathcal{GK}(\SM^n)$.

\begin{ath}\label{rigid} The submanifolds $\mathcal{K}_{\pm\frac12}(\SM^n)$ are connected components of $\mathcal{GK}(\SM^n)$.
\end{ath}

\begin{proof}
Let $\mathcal{M}$ be the connected component of $\mathcal{GK}(\SM^n)$ containing $\mathcal{K}_{\frac12}(\SM^n)$ and 
let $\Psi_t$ be a curve in $\mathcal{M}$ starting at some point of $\mathcal{K}_{\frac12}(\SM^n)$, i.e. a smooth 1-parameter family of spinors on $\SM^n$ satisfying 
\beq\label{gkst}
\nabla_X\Psi_t=A_t(X)\.\Psi_t,
\eeq 
where $A_t\in\Gamma(\End^+(\T \SM^n))$ is symmetric for all $t$ and $A_0=\tfrac12\id$. Without any loss in generality we can assume that $\Psi_t$ has unit norm for every $t$. We will denote the derivative with respect to $t$ by a dot and drop the subscript whenever the objects are evaluated at $t=0$. Differentiating \eqref{gkst} with respect to $t$ and evaluating at $t=0$ yields
\beq\label{gkstd}
\nabla_X\dot\Psi=\dot A(X)\.\Psi+\tfrac12X\cdot\dot\Psi.
\eeq 
Taking the covariant derivative in this equation and skew-symmetrizing gives
$$R_{Y,X}\dot\Psi=- [(\nabla_X \dot A)Y  - (\nabla_Y \dot A)X  ] \cdot \Psi+[ \dot A(X) \wedge Y+X\wedge\dot A(Y) ]\cdot \Psi+\tfrac12 X  \wedge Y \cdot \dot\Psi.$$
Using the fact that the spinorial curvature on the sphere satisfies $R_{Y,X}\Phi=\tfrac12 X\wedge Y\cdot\Phi$ for every spinor $\Phi$, the previous equation reads
\beq\label{adot} [(\nabla_X \dot A)Y  - (\nabla_Y \dot A)X  ] \cdot \Psi=[ \dot A(X) \wedge Y+X\wedge\dot A(Y) ]\cdot \Psi.\eeq

On the other hand, differentiating at $t=0$ the equation \eqref{three2} satisfied by $A_t$ yields
$$0=2(\tr A)(\tr\dot A)-2\tr (A\dot A)=(n-1)\tr\dot A,$$
whence $\dot A$ is trace-free at $t=0$. Moreover, from \eqref{three1} we also get $\delta^\nabla\dot A=0$.

We now use the fact that $|X\cdot \Phi|^2=|X|^2$ for every $X\in \T M$ and for every unit spinor $\Phi$, whereas $|\omega\cdot\Phi|^2\le |\omega|^2$ for $\omega\in\Lambda^2M$. From \eqref{adot} we thus get (using a local orthonormal basis $e_i$ of the tangent bundle):
\bea|d^\nabla \dot A|^2&=&\tfrac12\sum_{i,j=1}^n|(\nabla_{e_i} \dot A)e_j  - (\nabla_{e_j} \dot A)e_i|^2\le\tfrac12 \sum_{i,j=1}^n| \dot A(e_i) \wedge e_j+e_i\wedge\dot A(e_j)|^2\\
&=&(n-1)|\dot A|^2+ \sum_{i,j=1}^ng(\dot A(e_i) \wedge e_j,e_i\wedge\dot A(e_j))=(n-2)|\dot A|^2.
\eea

Recall now the Weitzenb\"ock formula for trace-free symmetric tensors $h$ (cf. \cite[Prop. 4.1]{bour}):
\beq\label{w}(d^\nabla\delta^\nabla+\delta^\nabla d^\nabla)h=\nabla^*\nabla h+ h\circ\Ric-\mathring{R}(h),\eeq
where 
$$\mathring{R}(h)(X):=\sum_{i=1}^nR_{X,h(e_i)}e_i$$
(note that there is a sign change between Bourguignon's and our curvature convention). On the round sphere $\SM^n$ we have $\Ric=(n-1)\id$ and $\mathring{R}(h)(X)=-h(X)$. Applying \eqref{w} to $h:=\dot A$ and using the relation above $\delta^\nabla \dot A=0$, we get
$$\delta^\nabla d^\nabla \dot A=\nabla^*\nabla \dot A+n\dot A.$$
Taking the scalar product with $\dot A$ and integrating over $\SM^n$ (whose volume element is denoted by $\vol$) yields 
$$\int_{\SM^n}|d^\nabla \dot A|^2\vol=\int_{\SM^n}\left(|\nabla \dot A|^2+n|\dot A|^2\right)\vol,$$
which together with the previous inequality $|d^\nabla \dot A|^2\le (n-2)|\dot A|^2$ implies $\dot A=0$.

Going back to \eqref{gkstd} we thus see that $\dot\Psi$ is a Killing spinor. In other words, we have shown that for every $\Psi\in \mathcal{K}_{\frac12}(\SM^n)$, the tangent space $\T_\Psi \mathcal{M}$ is contained in $\mathcal{K}_{\frac12}(\SM^n)$. This shows that $\mathcal{M}=\mathcal{K}_{\frac12}(\SM^n)$. The proof of the statement for $\mathcal{K}_{-\frac12}(\SM^n)$ is similar.
\end{proof}

\section{Appendix. The canonical spinor on $3$-Sasakian manifolds of dimension $7$}

We give here an alternative definition of the canonical spinor on 3-Sasakian 7-dimensional manifolds discovered by Agricola and Friedrich \cite{af10}. This approach makes use of the Riemannian cone construction which we now recall. 

The Riemannian cone over $(M,g)$ is the Riemannian manifold $(\bar M,\bar g):=(\RM^+\times M,dt^2+t^2g)$. The radial vector $\xi:=t\tfrac{\partial}{\partial t}$ satisfies the equation 
\beq\label{dxi} \bar\nabla_X\xi=X,\qquad\forall\ X\in \T\bar M,
\eeq
where $\bar\nabla$ denotes the Levi-Civita covariant derivative of $\bar g$. Assume now that $M$ is 3-Sasakian. It is well known (and is nowadays the standard definition of 3-Sasakian structures) that $\bar M$ has a hyperk\"ahler structure $J_1,J_2,J_3$, such that the vector fields $\xi_i:=J_i(\xi)$  on $\bar M$ are Killing and tangent to the hypersurfaces $M_t:=\{t\}\times M$. When restricted to $M=M_1$, $\xi_i$ are unit Killing vector fields satisfying the 3-Sasakian relations. 

Suppose now that $M$ has dimension $7$. The real spin bundle of $M$ is canonically identified with the positive spin bundle $\Sigma^+\bar M$ restricted to $M_1=M$. With respect to this identification, if $\psi\in\Gamma(\Sigma M)$ is the restriction to $M$ of a spinor $\Psi\in\Gamma(\Sigma^+\bar M)$ and $X$ is any vector field on $M$ identified with a vector field on $\bar M$ along $M_1$, then
\beq\label{xp}X\cdot\psi=X\cdot\xi\cdot\Psi
\eeq
and
\beq\label{np}\nabla_X\psi=\bar\nabla_X\Psi+\tfrac12 X\cdot\xi\cdot\Psi.
\eeq
Recall now that the restriction to $\Sp(2)$ of the half-spin representation $\Sigma_8^+$ has a 3-dimensional trivial summand. Correspondingly, on $\bar M$ there exist three linearly independent $\bar\nabla$-parallel spinor fields on which every 2-form from $\sp(2)$ (i.e. commuting with $J_1$, $J_2$, $J_3$) acts trivially by Clifford multiplication. Moreover, there exists exactly one such unit spinor $\Psi_1$ (up to sign) on which the Clifford action of $\Omega_1$ (the K\"ahler form of $J_1$) is also trivial (cf. \cite{wang}). 
\begin{elem}
The spinor $\Psi_0:=\tfrac1{|\xi|^2}\xi\cdot\xi_1\cdot\Psi_1$ satisfies 
\beq\label{bar}\bar\nabla_X\Psi_0=\bar A(X)\cdot\xi\cdot\Psi_0,
\eeq
where 
$$\bar A(X):=\begin{cases}\quad\  0 &\hbox{if $X$ belongs to the distribution $D:=\la\xi,\xi_1,\xi_2,\xi_3\ra$}\\-2X &\hbox{if $X\in D^\perp$.}
\end{cases}$$
\end{elem}
\begin{proof}
Since $J_i$ are $\bar\nabla$-parallel, \eqref{dxi} yields $\bar\nabla_X\xi_i=J_i(X)$ for all $X\in\T\bar M$. We thus have 
\beq\label{der}\bar\nabla_X\Psi_0=\tfrac1{|\xi|^2}(X\cdot\xi_1\cdot\Psi_1+\xi\cdot J_1(X)\cdot\Psi_1)-\tfrac2{|\xi|^4}\bar g(\xi,X)\xi\cdot\xi_1\cdot\Psi_1.\eeq
This relation gives immediately $\bar\nabla_\xi\Psi_0=0$ and $\bar\nabla_{\xi_1}\Psi_0=0$. Moreover, since the 2-form $\xi\wedge\xi_1-\xi_2\wedge\xi_3$ commutes with $J_1,J_2,J_3$, it belongs to $\sp(2)$ and thus acts trivially by Clifford multiplication on $\Psi_1$. We then obtain $\xi\cdot\xi_1\cdot\Psi_0=\xi_2\cdot\xi_3\cdot\Psi_0$, which together with \eqref{der} yields $\bar\nabla_{\xi_2}\Psi_0=\bar\nabla_{\xi_3}\Psi_0=0$.

It remains to treat the case where $X$ is orthogonal to $\la\xi,\xi_1,\xi_2,\xi_3\ra$. Assume that $X$ is scaled to have unit norm. We consider the orthonormal basis of $\T\bar M$ at some point $x\in M_1$ given by $e_1=\xi,e_2=\xi_1,e_3=\xi_2,e_4=\xi_3,e_5=X,e_6=J_1(X),e_7=J_2(X),e_8=J_3(X)$. 
Since $\Omega_1\cdot\Psi_1=0$ where $\Omega_1=e_1\cdot e_2+e_3\cdot e_4+e_5\cdot e_6+e_7\cdot e_8$, we obtain $\Omega_1\cdot\Psi_0=0$.
Now, the 2-form $e_5\wedge e_6-e_7\wedge e_8$ belongs to $\sp(2)$ and its Clifford action commutes with $e_1\cdot e_2$, thus 
$e_5\cdot e_6\cdot\Psi_0=e_7\cdot e_8\cdot\Psi_0$. Together with the relation $e_1\cdot e_2\cdot\Psi_0=e_3\cdot e_4\cdot\Psi_0$ proved above and the fact that 
$\Omega_1\cdot\Psi_0=0$, we get
\beq (e_1\cdot e_2+e_5\cdot e_6)\cdot\Psi_0=0.\eeq
Using \eqref{der} we then compute at $x$:
\bea \bar\nabla_X\Psi_0&=&(X\cdot\xi_1\cdot\Psi_1+\xi\cdot J_1(X)\cdot\Psi_1)=(e_5\cdot e_2+e_1\cdot e_6)\cdot (-e_1\cdot e_2\cdot\Psi_0)\\
&=&(e_1\cdot e_5+e_2\cdot e_6)\cdot \Psi_0=e_5\cdot e_2\cdot(e_1\cdot e_2+e_5\cdot e_6)\cdot\Psi_0+2e_1\cdot e_5\cdot\Psi_0\\
&=&2e_1\cdot e_5\cdot\Psi_0=-2X\cdot\xi\cdot \Psi_0,
\eea
thus proving the lemma.
\end{proof}
As a direct consequence of this result, together with \eqref{xp}--\eqref{np}, we obtain the following:
\begin{ecor}[\cite{af10}, Thm. 4.1] The spinor $\psi_0:=\Psi_0\res_M$ is a generalized Killing spinor on $M$ satisfying 
\beq\label{bar1}\nabla_X\psi_0=A(X)\cdot\psi_0,
\eeq
where 
$$A(X):=\begin{cases}\quad\tfrac12X& \hbox{if $X$ belongs to the distribution $D:=\la\xi_1,\xi_2,\xi_3\ra$}\\ \; -\tfrac32X& \hbox{if $X\in D^\perp$.}
\end{cases}$$
\end{ecor}

\end{document}